\documentclass[11pt]{amsproc}

\usepackage[top=1in,bottom=1in,left=1.5in,right=1.5in]{geometry}
\usepackage{amsfonts} 
\usepackage{amsmath}
\usepackage{amssymb}
\usepackage{amsthm}
\usepackage{setspace}
\usepackage{mathrsfs}
\usepackage{subfigure}
\usepackage{url}
\usepackage{booktabs}
\usepackage{ifthen}
\usepackage{tikz}
\usetikzlibrary{calc}
\usepackage{enumerate}

\newcommand{\cP}{\mathcal{P}}

\newtheorem{thm}{Theorem}[section]
\newtheorem{lem}[thm]{Lemma}
\newtheorem{prop}[thm]{Proposition}
\newtheorem{cor}[thm]{Corollary}
\newtheorem{claim}[thm]{Claim}
\newtheorem{obs}[thm]{Observation}
\newtheorem{case}{Case}
\newtheorem{alg}{Algorithm}

\newtheorem{defn}[thm]{Definition}

\newtheorem*{conja}{Conjecture 1}

\numberwithin{equation}{section}

\begin{document}

\title{Grundy domination of forests and the strong product conjecture}

\author{Kayla Bell}
\address{Kayla Bell (\tt kbell28@student.clayton.edu)}

\author{Keith Driscoll}
\address{Keith Driscoll (\tt keithdriscoll@clayton.edu)}

\author{Elliot Krop}
\address{Elliot Krop (\tt elliotkrop@clayton.edu)}

\author{Kimber Wolff}
\address{Kimber Wolff (\tt kwolff1@student.clayton.edu)}
\address{Department of Mathematics, Clayton State University}

\date{\today}

\maketitle

\begin{abstract}
A maximum sequence $S$ of vertices in a graph $G$, so that every vertex in $S$ has a neighbor which is independent, or is itself independent, from all previous vertices in $S$, is called a Grundy dominating sequence. The Grundy domination number, $\gamma_{gr}(G)$, is the length of $S$. We show that for any forest $F$, $\gamma_{gr}(F)=|V(T)|-|\cP|$ where $\cP$ is a minimum partition of the non-isolate vertices of $F$ into caterpillars in which if two caterpillars of $\cP$ have an edge between them in $F$, then such an edge must be incident to a non-leaf vertex in at least one of the caterpillars. We use this result to show the strong product conjecture of B.~Bre\v{s}ar, Cs.~Bujt\'{a}s, T.~Gologranc, S.~Klav\v{z}ar, G.~Ko\v{s}mrlj, B.~Patk\'{o}s, Zs.~Tuza, and M.~Vizer, \emph{Dominating sequences in grid-like and toroidal graphs}, Electron. J. Combin. 23(4): P4.34 (2016), for all forests. Namely, we show that for any forest $G$ and graph $H$, $\gamma_{gr}(G \boxtimes H) = \gamma_{gr}(G) \gamma_{gr}(H)$. We also show that every connected graph $G$ has a spanning tree $T$ so that $\gamma_{gr}(G)\le \gamma_{gr}(T)$ and that every non-complete connected graph contains a Grundy dominating set $S$ so that the induced subgraph of $S$ contains no isolated vertices. 
\\[\baselineskip] 2020 Mathematics Subject
      Classification: 05C69, 05C76
\\[\baselineskip]
\end {abstract}

\section{Introduction}

For any graph $G$, a sequence of vertices $S=(v_1,\dots, v_k)$ is called a \emph{legal sequence} if for every $i\in [k]$, $N[v_i]- \bigcup_{j=1}^{i-1}N[v_j]\neq \emptyset$. A longest legal sequence is called a \emph{Grundy dominating sequence} of $G$ and the size of such a sequence is called the \emph{Grundy domination number} and is denoted by $\gamma_{gr}(G)$.

Grundy domination was introduced in \cite{BGMRR} several years ago, inspired by the much studied game domination number and the domination game. Since then, a multitude of papers have been published on aspects of this function and its variants, e.g. \cite{BHR}, \cite{BBGKKPTV1}, \cite{BBGKKPTV2}, \cite{BBGKKMPTV}. 

A principal direction of inquiry in the original paper \cite{BGMRR} was to understand $\gamma_{gr}(T)$ for any tree $T$. To describe the results of that paper, call $ES(T)$ the \emph{end support vertices} of $T$ which are support vertices (non-leaves adjacent to a leaf) each of which are adjacent to at most one non-leaf vertex of $T$. For a lower bound, the authors produced an algorithm to find a legal sequence of $T$ of length at least $|V(T)|-|ES(T)|+1$. To find an upper bound, the authors defined an equivalence relation between end support vertices with equivalence classes $\widetilde{T}$. They then proved that the Grundy domination number does not exceed $|V(T)|-|ES(T)|+|\widetilde{T}|$. To summarize, the authors showed that for any tree $T$,

\begin{align}\label{treeineq}
|V(T)|-|ES(T)|+1\le \gamma_{gr}(T) \le |V(T)|-|ES(T)|+|\widetilde{T}|.
\end{align}

Another collection of natural questions for domination functions concern their behavior on graph products, inspired by the famous Vizing's conjecture \cite{V1}, which states that for any graphs $G$ and $H$, if $\square$ is the Cartesian product of graphs,

\begin{align}\label{V}
\gamma(G\square H)\ge \gamma(G)\gamma(H).
\end{align}

In \cite{BBGKKPTV1}, the authors investigated relations between the Grundy domination number of various graph products as they related to the Grundy domination numbers of the factor graphs. An outstanding conjecture concerned the strong product of graphs $\boxtimes$. For any graphs $G$ and $H$, the strong product $G\boxtimes H$ is the graph on the vertices in $(g,h)$ for every $g\in V(G)$ and $h\in V(H)$. Any pair of vertices $(g_1,h_1)$ and $(g_2,h_2)$ are adjacent if either 
\begin{enumerate}
\item $g_1=g_2$ and $h_1$ is adjacent to $h_2$ in $H$, or
\item $g_1$ is adjacent to $g_2$ in $G$ and $h_1 = h_2$, or
\item $g_1$ is adjacent to $g_2$ in $G$ and $h_1$ is adjacent to $h_2$ in $H$.
\end{enumerate}

The authors easily showed that for any graphs $G$ and $H$, 

\begin{align}\label{ineq1}
\gamma_{gr}(G \boxtimes H) \ge \gamma_{gr}(G)\gamma_{gr}(H)
\end{align}

and posed

\begin{conja}\label{conj}
For any graphs $G$ and $H$,
\[\gamma_{gr}(G \boxtimes H) = \gamma_{gr}(G)\gamma_{gr}(H).\]
\end{conja}

Among other results, the authors showed this conjecture holds if $G$ is a caterpillar tree and $H$ is any graph.

In this paper, we find an exact formula for the Grundy domination number of forests. In Section $3$, we show that for any forest $F$, $\gamma_{gr}(F)=|V(F)|-|\cP|$ where $\cP$ is a minimum partition of the non-isolate vertices of $F$ into caterpillars in which if two caterpillars of $\cP$ have an edge between them in $F$, then such an edge must be incident to a non-leaf vertex in at least one of the caterpillars. In Section $4$ we use our result for forests to show Conjecture \ref{conj} true when $G$ is a forest. In section $5$ we show that every connected graph $G$ has a spanning tree with Grundy domination number at least as large as $\gamma_{gr}(G)$. We also show that every non-complete connected graph contains a Grundy dominating set $S$ so that the induced subgraph of $S$ contains no isolated vertices.

\section{More definitions and known results}

We follow established notation from \cite{BBGKKPTV1} which can also be found in other papers on the subject. For any legal sequence $S=(v_1,\dots, v_k)$ we call the set of vertices composed of the vertices from the sequence a \emph{legal set} and write $\widehat{S}=\{v_1,\dots, v_k\}$. For any $i\in [k]$, we say that $v_i$ \emph{footprints} the vertices $N[v_i]- \bigcup_{j=1}^{i-1}N[v_j]$ and that the vertices of $N[v_i]- \bigcup_{j=1}^{i-1}N[v_j]$ are \emph{footprinted} by $v_i$ or the \emph{footprint} of $v_i$. Also, we say that $v_i$ is the \emph{footprinter} of $N[v_i]- \bigcup_{j=1}^{i-1}N[v_j]$.

\medskip

For notational convenience, when producing a legal sequence for a graph $G$, we identify \emph{labels} on vertices of $G$ with the indices of the legal sequence. That is, a label on a vertex will indicate the sequential position of that vertex.

\medskip

Let $G$ and $H$ be arbitrary graphs. For any $h\in V(H)$, define the \emph{$G$-fiber} of $G\boxtimes H$ with respect to $h$ as the induced subgraph on $\{(g,h)\in G\boxtimes H: g\in V(G)\}$ and denote it by $G^h$. Similarly, for any $g\in V(G)$, define the \emph{$H$-fiber} of $G\boxtimes H$ with respect to $g$ as the induced subgraph on $\{(g,h)\in G\boxtimes H: h\in V(H)\}$ and denote it by $H^g$.

\medskip

A tree $C$ is called a \emph{caterpillar} if for a maximum path $P$ of $C$, every vertex is of distance at most $1$ from $P$. A consequence of inequalities \eqref{treeineq} from \cite{BGMRR}, is that if a tree $C$ is a caterpillar, then $\gamma_{gr}(C)=|V(G)|-1$. 

\medskip

A vertex is called \emph{simplicial} if its neighbors form a clique. The following result was shown in \cite{BBGKKPTV1} but we state the proof here for completeness.

\begin{prop}{\cite{BBGKKPTV1}}\label{simplicial}
For any graphs $G$ and $H$, if $v$ is a simplicial vertex of $G$, then
\[\gamma_{gr}(G \boxtimes H) \le \gamma_{gr}(H) + \gamma_{gr}((G-v)\boxtimes H)\]
\end{prop}
\begin{proof}
Among all Grundy dominating sequences of $G\boxtimes H$, let $D$ be one that has the maximum number of vertices from $H^v$. Let $D_1$ be the subsequence of $D$ consisting of the vertices in $\widehat{D}\cap H^v$  and $D_2=D- D_1$. Notice that by projecting $\widehat{D}_1$ onto $H$ we obtain vertices that form a legal sequence of $H$, taken in the same order as $D_1$. Call this projected sequence $p_H(D_1)$. To see this, assume that a vertex $(v,y)\in \widehat{D}_1$ footprints a vertex $(v',y')\in G\boxtimes H$ and notice that in $H$, $y$ must footprint $y'$ with respect to the sequence $p_H(D_1)$. Thus, $|\widehat{D}_1|\le \gamma_{gr}(H)$.

Notice that if $u$ is a neighbor of $v$ in $G$ and for some $h_1$ in $H$, $(u,h_1)$ footprints $(v,h_2)$ for some $h_2$ in $H$, then since $v$ is simplicial, $(u,h_1)$ may be replaced by $(v,h_1)$ in $D$ to produce a legal sequence with more vertices from $H^v$ than $D$. This contradicts the maximum choice of $D$. Hence, no vertex of $\widehat{D}_2$ footprints a vertex in $H^v$. This means that $D_2$ is a legal sequence of $(G-v)\boxtimes H$. Thus, we have the desired inequality
\[\gamma_{gr}(G\boxtimes H)=|\widehat{D}_1|+|\widehat{D}_2|\le \gamma_{gr}(H)+\gamma_{gr}((G-v)\boxtimes H).\]
\end{proof}

A vertex $u$ is called a \emph{twin vetex} if there exists another vertex $v$ so that $N[u]=N[v]$. The next three results were shown in \cite{BGK}.

\begin{thm}{\cite{BGK}}\label{edgedel}
If $G$ is a graph and $e\in E(G)$, then 
\[\gamma_{gr}(G)-1\le \gamma_{gr}(G-e)\le \gamma_{gr}(G)+1.\]
Moreover, there exist graphs $G$ such that all values of $\gamma_{gr}(G-e)$ between $\gamma_{gr}(G)-1$ and $\gamma_{gr}(G)+1$ are realized for different edges $e\in E(G)$.
\end{thm}

\begin{thm}{\cite{BGK}}\label{vertdel}
If $G$ is a graph and $v\in V(G)$, then 
\[\gamma_{gr}(G)-2\le \gamma_{gr}(G-v)\le \gamma_{gr}(G).\]
Moreover, there exist graphs $G$ such that all values of $\gamma_{gr}(G-v)$ between $\gamma_{gr}(G)-2$ and $\gamma_{gr}(G)$ are realized for different vertices $v\in V(G)$.
\end{thm}

\begin{prop}{\cite{BGK}}\label{simp}
Let $G$ be a graph and $u\in V(G)$.
\begin{enumerate}
\item If $u$ is a simplicial vertex, then $\gamma_{gr}(G-u)\ge \gamma_{gr}(G)-1$.
\item If $u$ is a twin vertex, then $\gamma_{gr}(G-u)=\gamma_{gr}(G)$.
\end{enumerate}
\end{prop}

\section{Grundy domination of forests}
Let $C$ be a caterpillar with more than one vertex. Choose a path $P=\{v_1,\dots, v_k\}$ of maximum length with leaves $v_1,v_k$ and the rest of the vertices of $P$ non-leaf vertices. For $i\in \{2,\dots, k-1\}$, let $L_i$ indicate the set of leaf neighbors of $v_i$. We now define a labeling on $C$ which we will show produces a legal sequence of $C$.

\medskip

\begin{alg}[Caterpillar Labeling]

Starting with $v_2$, label the vertices in $L_2$ by consecutive integers starting from $1$. If $k=2$, then stop. Otherwise, label $v_2$ by the next consecutive integer. Continue labeling $L_3$ by the next consecutive integers. If $k=3$, then stop. Otherwise, label $v_3$ by the next consecutive integer. Repeat these steps for all non-leaf vertices in order along $P$, ending at $v_{k-1}$ but in the last step label all but one vertex of $L_{k-1}$ and then label $v_{k-1}$.

\end{alg}

\begin{prop}\label{cat}
For any caterpillar $C$, the Caterpillar Labeling produces a length $|C|-1$ legal sequence of $C$.
\end{prop}

\begin{proof}
Notice that every leaf vertex chosen in the Caterpillar Labeling footprints itself. For $i\in \{2,\dots, k-2\},\,v_i$ footprints $v_{i+1}$. Finally, $v_{k-1}$ footprints the unlabeled leaf in $L_{k-1}$.
\end{proof}

\medskip

For any forest $F$, a \emph{minimum caterpillar partition} of $F$, $\cP$, is a partition of the non-isolate vertices of $F$ into sets $C_1,\dots, C_{\ell}$ so that
\begin{enumerate}
\item For any $i$, the induced subgraph on $C_i$ is a caterpillar.
\item If for some $i$ and $j$, there is an edge between a vertex of $C_i$ and $C_j$ in $T$, then that edge must be adjacent to a non-leaf vertex of the induced subgraph on $C_i$ or $C_j$.
\item $\cP$ is chosen to have the minimum number of caterpillars.
\end{enumerate}

\medskip

We may refer to a minimum caterpillar partition of a forest $F$ as $\cP(F)$. For any minimum caterpillar partition $\cP$ of a forest $F$ we call the edges of $F$ between caterpillars of $\cP$, \emph{branch edges} and the endpoints of branch edges, \emph{branch vertices}. We say that a caterpillar $C\in \cP$ is a \emph{leaf caterpillar} if it contains only one branch vertex. Two caterpillars in $\cP$ are \emph{neighbors} if there is a branch edge between them.

\medskip

We now introduce the useful concept of the structure produced by contracting caterpillars in a minimum caterpillar partition of a forest. For any forest $F$ with minimum caterpillar partition $\cP$, we define the \emph{canopy graph} as the graph $CP(F,\cP)$, or just $CP$ when clear from context, with vertex set corresponding to the caterpillars of $\cP$ contracted to vertices. If $\cP=\{C_1,\dots, C_{\ell}\}$, then we write $V(CP)=\{c_1,\dots, c_{\ell}\}$. Two vertices $c_i$ and $c_j$ of $CP$ are adjacent if the corresponding caterpillars $C_i$ and $C_j$ are neighbors.

\begin{obs}\label{canopy forest}
For any caterpillar partition $\cP$ of a forest $F$, $CP(F,\cP)$ is a forest.
\end{obs}

\begin{proof}
Notice that any cycle in $CP$ can be extended to a cycle in $F$, yielding a contradiction.
\end{proof}

\medskip

Suppose $\cP=\{C_1,\dots, C_{\ell}\}$ and that for any $i\in [\ell]$, $C_i$ contains $b(i)$ branch vertices. Let $L$ be a maximum path of $C_i$ represented from left to right. If $v$ is a vertex of $L$, we say that the \emph{position} of $v$ is $1$ when $v$ is the left-most vertex on $L$. If $v$ is some other vertex of $L$, then the position of $v$ is one plus the distance from the vertex of position $1$. If $v$ is a vertex not on $L$, then the position of $v$ is the distance of $v$ from the vertex of position $1$. Let us define the branch vertices of $C_i$ as $v_{i_1},\dots,v_{i_{b(i)}}$ where $i_j<i_{j'}$ when the position of $v_{i_j}$ is smaller than the position of $v_{i_{j'}}$. Define the \emph{rank} of a branch vertex $v_{i_j}$ as $j$.

In other words, the rank of a branch vertex on $C_i$ is one plus the number of branch vertices that preceded it (with respect to position) on $C_i$, when counting from left to right. We note here that there may be more than one branch vertex with the same position.

We say that an integer label is \emph{available} if it has not been used previously on a vertex.

\begin{lem}\label{11}
For any minimum caterpillar partition $\cP$ of a forest $F$, if $F$ contains a component which is not a caterpillar, then there exist distinct integers $i$ and $j$ between $1$ and $\ell$ such that $C_i$ and $C_j$ each contain vertices which are adjacent in $F$ and have rank $1$.
\end{lem}

\begin{proof}
Choose a caterpillar, say $C_1$, contained in a component which is not a caterpillar, and call the vertex on $C_1$ of rank $1$, $u_1$. If $u_1$ is adjacent to a vertex of rank $1$ on another caterpillar, the proof is complete. Otherwise, say $u_1$ is adjacent to a vertex on a caterpillar of $\cP - C_1$, $C'_2$. Call the vertex of rank $1$ on $C'_2$,  $u_2$. If $u_2$ is adjacent to a vertex of rank $1$ on another caterpillar, the proof is complete. Otherwise, say $u_2$ is adjacent to a vertex on a caterpillar of $\cP - C_1 - C'_2$, $C'_3$. Continue this process and notice that it must result in a vertex of rank $1$ adjacent to another vertex of rank $1$, since there are finitely many caterpillars in $\cP$.
\end{proof}

\begin{alg}[Forest Labeling] Let $F$ be a given non-trivial forest. For the forest $F$, choose a minimum caterpillar partition $\cP=\{C_1,\dots, C_{\ell}\}$. For every $i\in [\ell]$, define the branch vertices of $C_i$ as $v_{i_1},\dots,v_{i_{b(i)}}$. In the labeling that follows, we label vertices by consecutive integers, starting with $1$. Perform the following labeling for every $i\in [\ell]$.
\begin{enumerate}
\item Set $j=1$ and set $F_1=F$. 
\item For every $i$, perform the Caterpillar Labeling on $C_i$ on all non-labeled vertices up to the vertices of the same position as the branch vertex of rank $1$, except for the branch vertices on that position, if such a vertex exists. If no such vertex exists, perform the caterpillar labeling on the remaining non-labeled vertices of $C_i$.
\item For all branch vertices of rank $1$ that do not have the largest position of all vertices on the caterpillar to which they belong, label the branch vertices of rank $1$ which are adjacent to other branch vertices of rank $1$ by consecutive integers, starting with the smallest available label. Furthermore, when labeling consecutive branch vertices such that one is a leaf, label the leaf branch vertex before the non-leaf branch vertex.
\item Label every caterpillar which does not contain an unlabeled branch vertex by the Caterpillar Labeling
\item Remove all caterpillars which have all but one vertex labeled and then remove all remaining labeled vertices and their incident edges.
\item Let $j=j+1$ and call the remaining forest $F_j$, then repeat the labeling unless $F_j=\emptyset$.
\item Label all isolate vertices by the next consecutive available labels.
\end{enumerate}
\end{alg}

\begin{defn}
We say a forest $F$ is \emph{caterpillar-critical} if for any leaf edge $e$, $|\cP(F-e)|<|\cP(F)|$. 
\end{defn}

In other words, a forest is caterpillar-critical if the removal of any leaf edge produces a forest with fewer caterpillars in any minimum caterpillar partition.

\begin{lem}\label{cc}
Suppose $F$ is a caterpillar-critical forest with minimum caterpillar partition $\cP$. Then any leaf caterpillar $C\in \cP$ is one of the following graphs
\begin{enumerate}
\item $P_2$ with one branch vertex
\item $P_5$ with the vertex of position $3$ as the only branch vertex of $C$
\end{enumerate}
\end{lem}

\begin{proof}
Suppose $C$ is a leaf caterpillar of $\cP$ which contains a path 

\noindent $L=\{v_1,\dots, v_k\}$ of maximum length. 

We first note that $C$ can have no leaf edges of $F$ other than $v_1v_2$ and $v_{k-1}v_k$, since we can remove such an edge without reducing the number of caterpillars in a minimum caterpillar partition, contradicting the criticality of $F$. In other words, the removal of such a leaf edge does not allow for $C$ to combine with another caterpillar to form a new caterpillar. 

Next, we claim that only vertices of $L$ may be branch vertices. Indeed, if $v$ not on $L$ is a branch vertex, then for any $k\ge 3$, $v_1$ and $v_k$ are each of distance at least $2$ to $v$. Now the removal of either edge $v_1v_2$ or $v_{k-1}v_k$ does not reduce the number of caterpillars in a minimum caterpillar partition, contradicting the criticality of $F$. In other words, if we remove $v_1v_2$ or $v_{k-1}v_k$ from $C$, then the resulting caterpillar cannot be combined with one of its neighbors to create a new caterpillar, which is a contradiction.

Together, these two observations imply that $C$ is a path. Since $C$ is a leaf caterpillar, it must have exactly one branch vertex. Notice that no branch vertex $v$ of $C$ can be of distance more than $2$ from either $v_1$ (or $v_k$), since otherwise removing $v_1v_2$ ($v_{k-1}v_k$) does not reduce the number of caterpillars in a minimum caterpillar partition. This means that $k\le 5$. If $k$ is $3$ or $4$, then there are two vertices of distance at least $2$ to the neighboring caterpillar to $C$. In this case, the removal of a leaf edge from $C$ does not allow us to combine it with its neighbor to form a caterpillar, since otherwise we could have combined $C$ with its neighbor to form a caterpillar, contradicting the minimality of $\cP$. Again, this contradicts the criticality of $F$.
\end{proof}

\begin{lem}\label{ccleaf}
If $F$ is a caterpillar-critical forest so that every minimum caterpillar partition $\cP$ does not contain a leaf caterpillar that is $P_5$, then every leaf caterpillar in $\cP$ is a neighbor of a caterpillar that is $P_5$ with the central vertex as the only branch vertex.
\end{lem}

\begin{proof}
Let $M$ be a path of maximum length in the canopy graph $CP$. If $c_1$ is an end vertex of $M$, then by Lemma \ref{cc}, $c_1$ corresponds to a caterpillar of $F$ which is either $P_5$ or $P_2$ and we must assume the latter. Suppose further that $M=\{c_1,c_2,\dots, c_m\}$ with consecutive vertices as neighbors. Let $C_2$ be the caterpillar of $F$ which corresponds to the vertex $c_2$ in $CP$. If $C_2$ does not contain a leaf in $F$, then the leaves in the induced subgraph of $C_2$ in $F$, $[C_2]_F$, must be branch vertices in $F$. Since all caterpillars have at least two leaves, there are at least two such branch vertices. Call one such branch vertex $x$ and the other $y$, and call the caterpillar of $\cP$ containing a branch vertex adjacent to $x$, $C'$, and the one containing a branch vertex adjacent to $y$, $C''$. Notice that by maximality of $M$ and the fact that $c_1$ is an end vertex of $L$, there exists at most one neighbor of $c_2$ in $CP$, which is not a leaf of $CP$, otherwise $CP$ would have a longer path than $M$. This means that either $C'$ or $C''$ is a leaf caterpillar in $\cP$. Without loss of generality suppose that it is $C'$. Notice that $C'$ cannot be $P_2$ since in that case it could be combined with $C_2$ to produce a larger caterpillar, contradicting the minimality of $\cP$. Thus, by Lemma \ref{cc}, $C'$ must be $P_5$, which leads to a contradiction to the preclusion of $P_5$ leaf caterpillars. Thus, $C_2$ must contain at least one leaf of $F$.

Next, let $L=\{v_1, \dots, v_k\}$ be a path of maximum length of $C_2$ and notice as in the proof of Lemma \ref{cc}, since $F$ is caterpillar-critical, $C_2$ can have no leaf edges of $F$ other than $v_1v_2$ and $v_{k-1}v_k$. It is also easy to see that $k\ge 5$, since otherwise $C_2$ could be combined with $C_1$ to form a larger caterpillar, contradicting the minimality of the caterpillar partition $\cP$. Furthermore, we may argue as in the proof of Lemma \ref{cc} that only vertices of $L$ may be branch vertices since otherwise the removal of $v_1v_2$ or $v_{k-1}v_k$ would not reduce the size of $\cP$. This means that $C_2$ is a path.

By the maximality of $M$ in $CP$, we note that there can be no vertex in $CP$ of distance $2$ to $c_2$ which is not on $M$. This means that any neighbors of $c_2$ other than $c_1$ or $c_3$ must be leaves, and by assumption, correspond  to caterpillars of $F$ which are $P_2$. We now note that $v_3$ and $v_{k-3}$ in $C_2$ must be branch vertices since otherwise the removal of $v_1v_2$ or $v_{k-1}v_k$ would not reduce the size of $\cP$, contradicting the criticality of $F$. 

If $k>5$, then $C_3$ is either {\bf not adjacent} to $C_2$ by a branch edge to $v_3$ or to $v_{k-3}$. Without loss of generality, we will assume that $C_3$ is not adjacent to $C_2$ by a branch edge to $v_3$. Let $C$ be a caterpillar adjacent to $C_2$ by a branch edge to $v_3$. Then $C$ is a leaf caterpillar $P_2=\{u_1, u_2\}$ where $u_2$ is a branch vertex. Note now that we may produce a new minimum caterpillar partition $\cP'$ from $\cP$ by combining $C_1$ and $C_2$ into a leaf caterpillar $P_5$ with the remaining vertices of $C_2$ as its neighbor. That is, we form a leaf $P_5$ caterpillar from $\{v_1,v_2, v_3, u_2, u_1\}$, we define $C_2'$ as the caterpillar $\{v_4, \dots, v_k\}$, and we make no alterations to the rest of the caterpillars in $\cP$. Since $|\cP'|=|\cP|$, this contradicts our assumptions about $F$. Thus, we conclude that $k=5$ and notice that only $v_3$ can be a branch vertex, so the proof is complete.
\end{proof}

\begin{thm}\label{forest}
For any forest $F$, $\gamma_{gr}(F)= |V(F)|-|\cP|$.
\end{thm}

\begin{proof}
First, notice that by Lemma \ref{11}, every iteration of the Forest Labeling can be initiated. To prove that $\gamma_{gr}(T)\ge |V(T)|-|\cP|$ we argue that the Forest Labeling produces a legal sequence for any forest $F$ of length $|F|-|\cP|$. First, notice that step $(2)$ produces a legal sequence by the Caterpillar Labeling. In step $(3)$, notice that a branch vertex which is a leaf footprints itself, and that any branch vertex which is not a leaf footprints a vertex of position one larger on the same caterpillar. Step $(4)$ produces a legal sequence by the Caterpillar Labeling. Next, we note that while leaves may footprint themselves in the Caterpillar Labeling, non-leaf vertices footprint other vertices on the Caterpillar. This fact justifies the legality of iterating the Forest Labeling for increments of $j$, since if vertices that had been removed in the previous iteration footprinted vertices in a current iteration, the vertices in the current iteration do not only footprint themselves. Finally, notice that if $u$ is a branch vertex with the largest position on the corresponding caterpillar $C$, then $u$ is the only vertex of $C$ which is not in the Grundy dominating sequence.

\medskip

Next we show that 
\begin{align}
\gamma_{gr}(F)\le |V(F)|-|\cP|. \label{ub}
\end{align}

We induct on the size of $F$, $|E(F)|$. The statement is true for $P_2$ so we suppose that it holds for all forests with fewer than $m$ edges. Let $F$ be a forest of size $m$.

\begin{claim}\label{edgeremoval}
For any forest $F$, if $e$ is a leaf edge of $F$, then $\gamma_{gr}(F-e)\ge \gamma_{gr}(F)$.
\end{claim}

\begin{proof}
Let $S$ be a Grundy dominating sequence of $F$ and suppose that $e=xy$ with $x$ a leaf vertex. Notice that if both $x$ and $y$ are elements of $S$, then $x$ must be chosen in $S$ previous to $y$, else it would not have a footprint in $F$. Thus, $y$ must footprint some other vertex in $F$. This implies that we may choose $S$ for $F-e$, where $x$ footprints itself as an isolate. 

If $x$ is an element of $S$ but $y$ is not, then we may choose $S$ for $F-e$ and, if $y$ remains undominated, create the legal sequence $S'$ by choosing $y$ as a final vertex and adding it to $S$.

If $y$ is an element of $S$ but $x$ is not, then let $S''$ be $S$ with $y$ replaced by $x$. If $y$ remains undominated, create the legal sequence $S'''$ by choosing $y$ as a final vertex and adding it to $S''$.

Notice that in all of these instances, we produce a legal set of $F$ of size at least $|S|$. This completes the proof of the claim.
\end{proof}

Next, suppose that there exists a leaf edge $e$ so that the size of a minimum caterpillar partition of $F$ is the same as the size of a minimum caterpillar partition of $F-e$.

Let $\cP'$ be a minimum caterpillar partition of $F-e$. By Claim \ref{edgeremoval}, we write
\[\gamma_{gr}(F)\le \gamma_{gr}(F-e)\le |F-e|-|\cP'|\le |F|-|\cP|.\]

We are now left to assume that $F$ is caterpillar-critical. By Lemma \ref{cc}, every leaf caterpillar of $F$ is either $P_2$ or $P_5$. Let $C$ be a leaf caterpillar.

\medskip

Regardless of whether $C$ is $P_2$ or $P_5$, by Lemma \ref{cc}, it contains exactly one branch vertex which we will call $v$. Call the neighbor of $v$ on the neighboring caterpillar, $x$ and let $e=vx$.

We note that the size of any minimum caterpillar partition of $F$ is the same as the size of any minimum caterpillar partition of $F-e$. By Theorem \ref{edgedel}, we need only consider three possibilities, $\gamma_{gr}(F-e)=\gamma_{gr}(F)+1$, $\gamma_{gr}(F-e)=\gamma_{gr}(F)$, and $\gamma_{gr}(F-e)=\gamma_{gr}(F)-1$. Let $S$ be a maximum legal sequence of $F$. 

If there is a legal sequence $S'$ of $F-e$, so that $|S| = |S'|-1$, then by the induction hypothesis,
\[\gamma_{gr}(F)=|S| = |S'| - 1 \le \gamma_{gr}(F-e) - 1\le |F-e|-|\cP(F-e)|-1=|F|-|\cP(F)|-1\]
which which is a contradiction with the already proved direction of the theorem.

If there is a legal sequence $S'$ of $F-e$, so that $|S| = |S'|$, then by the induction hypothesis,
\[\gamma_{gr}(F)=|S| = |S'| \le \gamma_{gr}(F-e) \le |F-e|-|\cP(F-e)|=|F|-|\cP(F)|\]
which again proves the theorem.

\medskip

Thus, we restrict our attention to the case when $\gamma_{gr}(F)=\gamma_{gr}(F-e)+1$. Since by the induction hypothesis together with the lower bound, $\gamma_{gr}(F-e) = |F-e|-|\cP(F-e)|=|F|-|\cP(F)|$, this implies that 
\begin{align}\label{over}
\gamma_{gr}(F)=|F|-|\cP(F)|+1
\end{align}

Notice that if either $v$ is not in the footprint of $x$ or $x$ is not in the footprint $v$, then $S$ is still a legal sequence of $F-e$. The same conclusion is attained if neither $v$ nor $x$ belong to $\widehat{S}$. Also, if $x\in \widehat{S}$ and the footprint of $x$ contains $v$ and some other vertex of $F$, then $S$ is a legal sequence of $F-e$. The same is true if $v\in \widehat{S}$ and the footprint of $v$ contains $x$ and some other vertex of $F$.

\medskip

Next we consider the two possibilies for $C$. Assume $C=P_5$ and that it is composed of the vertices $v_1,v_2,v_3,v_4,v_5$ with leaves $v_1, v_5$, where vertices with consecutive indices are adjacent. This means that the branch vertex $v=v_3$.

Suppose first that $\{x\}$ is the footprint of $v_3$. Applying the Pigeonhole Principle to equation \eqref{over}, either $C$ contains $|V(C)|$ vertices of $\widehat{S}$ or $F-C$ contains at least $|V(F)|-|V(C)|-|\cP(F-C)|+1=|V(F)|-|V(C)|-|\cP(F)|+2$ vertices of $\widehat{S}$. 

Suppose $C$ contains $|V(C)|$ vertices of $\widehat{S}$. Notice that $v_1,v_5\in \widehat{S}$ and that $v_1$ and $v_5$ must come earlier in $S$ than $v_2$ and $v_4$. However, this leads to a contradiction since if $v_2$ came before $v_4$ in $S$, then $v_4$ cannot be in $S$, since it would not have a footprint, and if $v_4$ came before $v_2$, then $v_2$ could not be chosen in $S$. This means that $C$ cannot contain $|V(C)|$ vertices.

Suppose $F-C$ contains at least $|V(F)|-|V(C)|-|\cP(F)|+2$ vertices of $\widehat{S}$. Since the footprint of $v_3$ is $\{x\}$, we note that $\widehat{S}\cap (F-C)$ is a legal set for $F-C$. Since $|\cP(F-C)|=|\cP(F)|-1$, by the induction hypothesis we have that any legal set of $F-C$ must be of size at most $|V(F)|-|V(C)|-|\cP(F-C)|=|V(F)|-|V(C)|-|\cP(F)|+1$, which contradicts the assumption on the number of vertices of $\widehat{S}$ in $F-C$.

\medskip

Next, we suppose that $\{v_3\}$ is the footprint of $x$. Applying the Pigeonhole Principle, as previously, if we assume that $C$ contains $|V(C)|$ vertices of $\widehat{S}$, then $v_1$ and $v_5$ must be in $\widehat{S}$ and appear in $S$ before $v_2$ and $v_4$. Again, the appearance of $v_2$ in $S$ precludes $v_4$ and the appearance of $v_4$ precludes $v_2$.

This means that we may assume that $F-C$ contains at least $|V(F)|-|V(C)|-|\cP(F)|+2$ vertices of $\widehat{S}$. First, let us assume that $F-C$ contains exactly $|V(F)|-|V(C)|-|\cP(F)|+2$ vertices of $\widehat{S}$. Notice that since $\{v_3\}$ is the footprint of $x$, if $v_2\in \widehat{S}$, then $v_2$ footprints $v_1$ and so $v_1\notin \widehat{S}$. Similarly, if $v_4\in \widehat{S}$, then $v_4$ footprints $v_5$ and so $v_5\notin\widehat{S}$. This means that some two vertices of $C$ are not in $\widehat{S}$. However, now we have that
\[|S|\le |V(F)|-|V(C)|-|\cP(F)|+2 + |V(C)|-2 = |V(F)|-|\cP(F)|.\]

Next, assume that $F-C$ contains more than $|V(F)|-|V(C)|-|\cP(F)|+2$ vertices of $\widehat{S}$. We first consider the forest $F'=F-\{v_1,v_2,v_4,v_5\}$ and note that $v_3$ is a simplicial vertex in $F'$. Applying Proposition \ref{simp}, we note that $\gamma_{gr}(F-C)\ge \gamma_{gr}(F')-1$. Furthermore, note that $\gamma_{gr}(F')\ge |\widehat{S}\cap (F-C)|$ since $\widehat{S}\cap (F-C)$ is a legal set in $F'$. However, this leads to
\[\gamma_{gr}(F-C)\ge |\widehat{S}\cap (F-C)|-1> |V(F)|-|V(C)|-|\cP(F)|+1,\]
which contradicts the induction hypothesis that
\[\gamma_{gr}(F-C) \le |V(F)|-|V(C)|-|\cP(F-C) \le |V(F)-|V(C)|-|\cP(F)|+1.\]

\medskip

Finally, we may assume that $F$ contains no leaf caterpillars which are $P_5$ so that $C=P_2$ and that it is composed of the vertices $u_1$ and $u_2$ where $u_2$ is a branch vertex. By Lemma \ref{ccleaf}, $C$ is a neighbor of a caterpillar $C_1=P_5$ which is composed of the vertices $v_1,v_2,v_3,v_4,v_5$ with leaves $v_1, v_5$ and branch vertex $v_3$. If $v_3$ does not have a neighbor which is not on $C$ or $C_1$, then we say we are in case $(*)$, which implies that that the component of the canopy graph of $F$ contains only two vertices, caterpillars $P_2$ and $P_5$, and so $P_5$ is a leaf of the canopy graph, which has already been considered. Otherwise, if we are not in case $(*)$, let $x$ be the neighbor of $v_3$ which is not on $C$ or $C_1$ and let $e=v_3x$.

As in the previous case, we suppose first that $\{x\}$ is the footprint of $v_3$. Applying the Pigeonhole Principle to equation \eqref{over}, either $C\cup C_1$ contains at least $|V(C)|+|V(C_1)|-1$ vertices of $\widehat{S}$ or $F-C-C_1$ contains at least $|V(F)|-|V(C)|-|V(C_1)|-|\cP(F-C-C_1)|+1=|V(F)|-|V(C)|-|V(C_1)|-|\cP(F)|+3$ vertices of $\widehat{S}$. 

Suppose $C\cup C_1$ contains $|V(C)|+|V(C_1)|-1$ vertices of $\widehat{S}$. Notice that the footprinter of $v_3$ is either $v_2$, $v_4$, or $u_2$. If $u_2$ footprints $v_3$, then only one of $v_1$ or $v_2$ may belong to $\widehat{S}$, and similarly, only one of $v_4$ or $v_5$ may belong to $\widehat{S}$. This implies that $C\cup C_1$ may contain at most $|V(C)|+|V(C_1)|-2$ vertices of $\widehat{S}$. If $v_2$ footprints $v_3$, then only one of $v_4$ or $v_5$ may belong to $\widehat{S}$, and only one of $u_1$ or $u_2$ may belong to $\widehat{S}$. Again, this implies that $C\cup C_1$ may contain at most $|V(C)|+|V(C_1)|-2$ vertices of $\widehat{S}$. The assumption that $v_4$ footprints $v_3$ proceeds in the same way.

Suppose $F-C-C_1$ contains $|V(F)|-|V(C)|-|V(C_1)|-|\cP(F)|+3$ vertices of $\widehat{S}$. Since the footprint of $v_3$ is $\{x\}$, we note that $\widehat{S}\cap (F-C-C_1)$ is a legal set for $F-C-C_1$. Since $|\cP(F-C-C_1)|=|\cP(F)|-2$, by the induction hypothesis we have that any legal set of $F-C-C_1$ must be of size at most $|V(F)|-|V(C)|-|V(C_1)|-|\cP(F-C-C_1)|=|V(F)|-|V(C)|-|V(C_1)|+2$, which contradicts the assumption on the number of vertices of $\widehat{S}$ in $F-C-C_1$.

\medskip

Next, we suppose that $\{v_3\}$ is the footprint of $x$. Again, we apply the Pigeonhole Principle, assuming first that $C\cup C_1$ contains at least $|V(C)|+|V(C_1)|-1$ vertices of $\widehat{S}$. Since $\{v_3\}$ is the footprint of $x$, and not of $v_2$ or $v_4$, if $v_2$ is in $\widehat{S}$, then $v_2$ footprints $v_1$, and $v_1$ is not in $\widehat{S}$. Likewise, if $v_4$ is in $\widehat {S}$, then $v_4$ footprints $v_5$, and $v_5$ is not in $\widehat{S}$. However, this implies that $C\cup C_1$ may contain at most $|V(C)|+|V(C_1)|-2$ vertices of $\widehat{S}$. Note that this argument also holds if we are in case $(*)$.

This means that we are left with the case where $F-C-C_1$ contains at least $|V(F)|-|V(C)|-|V(C_1)|-|\cP(F)|+3$ vertices of $\widehat{S}$. First, assume that $F-C-C_1$ contains exactly $|V(F)|-|V(C)|-|V(C_1)|-|\cP(F)|+3$ vertices of $\widehat{S}$. Since $v_3$ is in the footprint of $x$ it is not footprinted by $v_2, v_4,$ or $u_2$. This means that if $v_2$ is in $\widehat{S}$, then it footprints $v_1$, if $v_4$ is in $\widehat{S}$, then it footprints $v_5$, and if $u_2$ is in $\widehat{S}$, then it footprints $u_1$. In each of these cases, the footprinted vertex is not in $\widehat{S}$. However, now we have $|\widehat{S}\cap (C\cup C_1)|\le |V(C)|+|V(C_1)|-3$, which leads to
\[|S|\le |V(F)|-|V(C)|-|V(C_1)|-|\cP(F)|+3 +|V(C)|+|V(C_1)|-3 = |V(F)|-|\cP(F)|.\]

Finally, we assume that $F-C-C_1$ contains more than $|V(F)|-|V(C)|-|V(C_1)| - |\cP(F)|+3$ vertices of $\widehat{S}$. We first consider forest $F'=F-\{v_1,v_2,v_4,v_5,u_1,u_2\}$ and note that $v_3$ is a simplicial vertex in $F'$. Applying Proposition \ref{simp}, we note that $\gamma_{gr}(F-C-C_1)\ge \gamma_{gr}(F')-1$. Furthermore, note that $\gamma_{gr}(F')\ge |\widehat{S}\cap(F-C-C_1)|$ since $\widehat{S}\cap(F-C-C_1)$ is a legal set in $F'$. However, this leads to
\[\gamma_{gr}(F-C-C_1)\ge|\widehat{S}\cap (F-C-C_1)|-1>|V(F)-|V(C)|-|V(C_1)|-|\cP(F)|+2,\]
which contradicts the induction hypothesis that
\begin{align*}
\gamma_{gr}(F-C-C_1)&\le|V(F)|-|V(C)|-|V(C_1)|-|\cP(F-C-C_1)|\\
&\le |V(F)|-|V(C)|-|V(C_1)|-|\cP(F)|+2.
\end{align*}

\end{proof}

We now show two examples of the Forest Labeling. In the first example we assume all caterpillars are paths so as to focus on running the algorithm through several iterations. In the second example, we allow for arbitrary caterpillars to illustrate the labeling of branch vertices which are leaves in a caterpillar. The marked edges in each of the following forests are the edges which can be removed to produce the caterpillars in the decomposition. The rank of vertices is shown inside branch vertices. The enclosed vertices are those which are removed at the end of the iteration for the given value of $j$.

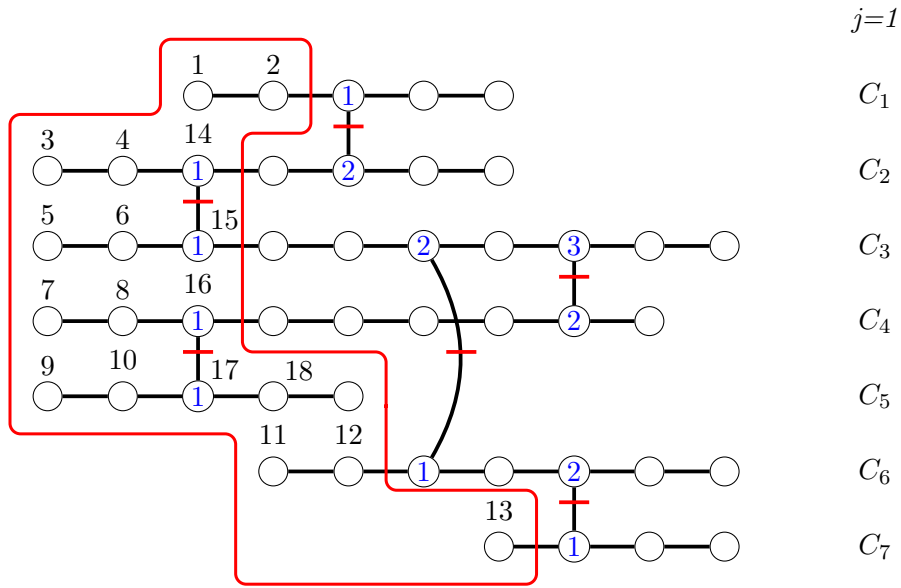
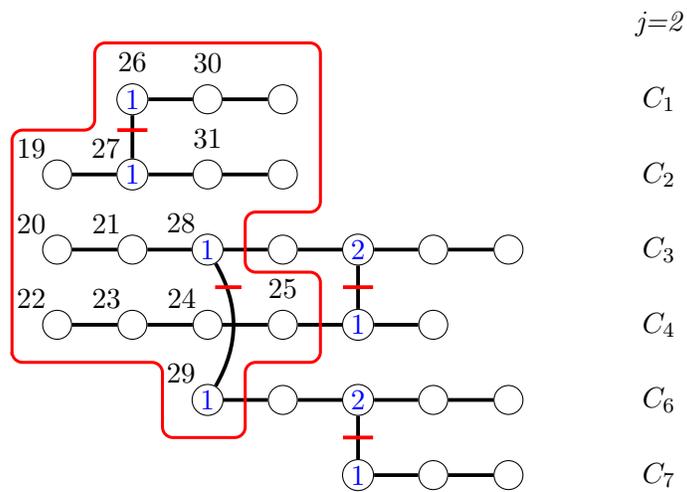
\begin{center}
\par\noindent\rule{\textwidth}{0.4pt}
\medskip
\begin{tikzpicture}
    \begin{scope}[auto, every node/.style={draw,circle,minimum size=1em,inner sep=1},node distance=1cm]
                \color{black}
                \tikzset{dot/.style={fill=black,circle}}

                    \node[shape=circle,draw=black,label=above: $1$] (1) at (3,7){};
                    \node[shape=circle,draw=black,label=above: $2$] (2) at (4,7){};
                    \node[shape=circle,draw=black] (c11) at (5,7){\textcolor{blue}{1}};
                    \node[shape=circle,draw=black] (c12) at (6,7){};
                    \node[shape=circle,draw=black] (c13) at (7,7){};

                    \draw[line width=.5mm,black,-] (1) edge[-] (2);
                    \draw[line width=.5mm,black,-] (2) edge[-] (c11);
                    \draw[line width=.5mm,black,-] (c11) edge[-] (c12);
                    \draw[line width=.5mm,black,-] (c12) edge[-] (c13);

                    \node[shape=circle,draw=black,label=above: $3$] (3) at (1,6){};
                    \node[shape=circle,draw=black,label=above: $4$] (4) at (2,6){};
                    \node[shape=circle,draw=black,label=above: $14$] (14) at (3,6){\textcolor{blue}{1}};
                    \node[shape=circle,draw=black] (c21) at (4,6){};
                    \node[shape=circle,draw=black] (c22) at (5,6){\textcolor{blue}{2}};
                    \node[shape=circle,draw=black] (c23) at (6,6){};
                    \node[shape=circle,draw=black] (c24) at (7,6){};

                    \draw[line width=.5mm,black,-] (3) edge[-] (4);
                    \draw[line width=.5mm,black,-] (4) edge[-] (14);
                    \draw[line width=.5mm,black,-] (14) edge[-] (c21);
                    \draw[line width=.5mm,black,-] (c21) edge[-] (c22);
                    \draw[line width=.5mm,black,-] (c22) edge[-] (c23);
                    \draw[line width=.5mm,black,-] (c23) edge[-] (c24);

                     \node[shape=circle,draw=black,label=above: $5$] (5) at (1,5){};
                     \node[shape=circle,draw=black,label=above: $6$] (6) at (2,5){};
                     \node[shape=circle,draw=black,label=above right: $15$] (15) at (3,5){\textcolor{blue}{1}};
                     \node[shape=circle,draw=black] (c31) at (4,5){};
                     \node[shape=circle,draw=black] (c32) at (5,5){};
                     \node[shape=circle,draw=black] (c33) at (6,5){\textcolor{blue}{2}};
                     \node[shape=circle,draw=black] (c34) at (7,5){};
                     \node[shape=circle,draw=black] (c35) at (8,5){\textcolor{blue}{3}};
                     \node[shape=circle,draw=black] (c36) at (9,5){};
                     \node[shape=circle,draw=black] (c37) at (10,5){};

                     \draw[line width=.5mm,black,-] (5) edge[-] (6);
                     \draw[line width=.5mm,black,-] (6) edge[-] (15);
                     \draw[line width=.5mm,black,-] (15) edge[-] (c31);
                     \draw[line width=.5mm,black,-] (c31) edge[-] (c32);
                     \draw[line width=.5mm,black,-] (c32) edge[-] (c33);
                     \draw[line width=.5mm,black,-] (c33) edge[-] (c34);
                     \draw[line width=.5mm,black,-] (c34) edge[-] (c35);
                     \draw[line width=.5mm,black,-] (c35) edge[-] (c36);
                     \draw[line width=.5mm,black,-] (c36) edge[-] (c37);

                     \node[shape=circle,draw=black,label=above: $7$] (7) at (1,4){};
                     \node[shape=circle,draw=black,label=above: $8$] (8) at (2,4){};
                     \node[shape=circle,draw=black,label=above: $16$] (16) at (3,4){\textcolor{blue}{1}};
                     \node[shape=circle,draw=black] (c41) at (4,4){};
                     \node[shape=circle,draw=black] (c42) at (5,4){};
                     \node[shape=circle,draw=black] (c43) at (6,4){};
                     \node[shape=circle,draw=black] (c44) at (7,4){};
                     \node[shape=circle,draw=black] (c45) at (8,4){\textcolor{blue}{2}};
                     \node[shape=circle,draw=black] (c46) at (9,4){};

                     \draw[line width=.5mm,black,-] (7) edge[-] (8);
                     \draw[line width=.5mm,black,-] (8) edge[-] (16);
                     \draw[line width=.5mm,black,-] (16) edge[-] (c41);
                     \draw[line width=.5mm,black,-] (c41) edge[-] (c42);
                     \draw[line width=.5mm,black,-] (c42) edge[-] (c43);
                     \draw[line width=.5mm,black,-] (c43) edge[-] (c44);
                     \draw[line width=.5mm,black,-] (c44) edge[-] (c45);
                     \draw[line width=.5mm,black,-] (c45) edge[-] (c46);

                     \node[shape=circle,draw=black,label=above: $9$] (9) at (1,3){};
                     \node[shape=circle,draw=black,label=above: $10$] (10) at (2,3){};
                     \node[shape=circle,draw=black,label=above right: $17$] (17) at (3,3){\textcolor{blue}{1}};
                     \node[shape=circle,draw=black,label=above right: $18$] (18) at (4,3){};
                     \node[shape=circle,draw=black] (c51) at (5,3){};

                     \draw[line width=.5mm,black,-] (9) edge[-] (10);
                     \draw[line width=.5mm,black,-] (10) edge[-] (17);
                     \draw[line width=.5mm,black,-] (17) edge[-] (18);
                     \draw[line width=.5mm,black,-] (18) edge[-] (c51);

                     \node[shape=circle,draw=black,label=above: $11$] (11) at (4,2){};
                     \node[shape=circle,draw=black,label=above: $12$] (12) at (5,2){};
                     \node[shape=circle,draw=black] (c61) at (6,2){\textcolor{blue}{1}};
                     \node[shape=circle,draw=black] (c62) at (7,2){};
                     \node[shape=circle,draw=black] (c63) at (8,2){\textcolor{blue}{2}};
                     \node[shape=circle,draw=black] (c64) at (9,2){};
                     \node[shape=circle,draw=black] (c65) at (10,2){};

                     \draw[line width=.5mm,black,-] (11) edge[-] (12);
                     \draw[line width=.5mm,black,-] (12) edge[-] (c61);
                     \draw[line width=.5mm,black,-] (c61) edge[-] (c62);
                     \draw[line width=.5mm,black,-] (c62) edge[-] (c63);
                     \draw[line width=.5mm,black,-] (c63) edge[-] (c64);
                     \draw[line width=.5mm,black,-] (c64) edge[-] (c65);

                     \node[shape=circle,draw=black,label=above: $13$] (13) at (7,1){};
                     \node[shape=circle,draw=black] (c71) at (8,1){\textcolor{blue}{1}};
                     \node[shape=circle,draw=black] (c72) at (9,1){};
                     \node[shape=circle,draw=black] (c73) at (10,1){};

                     \draw[line width=.5mm,black,-] (13) edge[-] (c71);
                     \draw[line width=.5mm,black,-] (c71) edge[-] (c72);
                     \draw[line width=.5mm,black,-] (c72) edge[-] (c73);

                    \draw[line width=.5mm,black,-] (c11) edge[-] (c22);
                    \draw[line width=.5mm,black,-] (14) edge[-] (15);
                    \draw[line width=.5mm,black,-] (c33) edge[bend left,-] (c61);
                    \draw[line width=.5mm,black,-] (c35) edge[-] (c45);
                    \draw[line width=.5mm,black,-] (16) edge[-] (17);
                    \draw[line width=.5mm,black,-] (c63) edge[-] (c71);

                \draw [rounded corners,line width=.4mm,color=red](.5,2.5) -- (.5,6.75)
                      [rounded corners] -- (2.5,6.75) -- (2.5,7.75)
                      [rounded corners] -- (4.5,7.75) -- (4.5,6.5)
                      [rounded corners] -- (3.59,6.5) -- (3.59,3.59)
                      [rounded corners] -- (5.5,3.59) -- (5.5,2.75)
                      [rounded corners] -- (5.5,2.75) -- (5.5,1.75)
                      [rounded corners] -- (7.5,1.75) -- (7.5,.5)
                      [rounded corners] -- (3.5,.5) -- (3.5,2.5) [rounded corners] -- cycle;

                    \draw (4.8,6.59) edge[line width=.5mm,red,-] (5.2,6.59);
                    \draw (2.8,5.59) edge[line width=.5mm,red,-] (3.2,5.59);
                    \draw (2.8,3.59) edge[line width=.5mm,red,-] (3.2,3.59);
                    \draw (6.3,3.59) edge[line width=.5mm,red,-] (6.7,3.59);
                    \draw (7.8,4.59) edge[line width=.5mm,red,-] (8.2,4.59);
                    \draw (7.8,1.59) edge[line width=.5mm,red,-] (8.2,1.59);

      \end{scope}

            \node[] at (12,7){$C_1$};
            \node[] at (12,6){$C_2$};
            \node[] at (12,5){$C_3$};
            \node[] at (12,4){$C_4$};
            \node[] at (12,3){$C_5$};
            \node[] at (12,2){$C_6$};
            \node[] at (12,1){$C_7$};

			\node[] at(12,8){$\emph{j=1}$};
    \end{tikzpicture}


\medskip
\par\noindent\rule{\textwidth}{0.4pt}
\medskip

    \begin{tikzpicture}
        \begin{scope}[auto, every node/.style={draw,circle,minimum size=1em,inner sep=1},node distance=1cm]
                    \color{black}
                    \tikzset{dot/.style={fill=black,circle}}

                        \node[shape=circle,draw=black,label=above: $26$] (26) at (2,6){\textcolor{blue}{1}};
                        \node[shape=circle,draw=black,label=above: $30$] (30) at (3,6){};
                        \node[shape=circle,draw=black] (c11) at (4,6){};

                        \draw[line width=.5mm,black,-] (26) edge[-] (30);
                        \draw[line width=.5mm,black,-] (30) edge[-] (c11);

                        \node[shape=circle,draw=black,label=above left: $19$] (19) at (1,5){};
                        \node[shape=circle,draw=black,label=above left: $27$] (27) at (2,5){\textcolor{blue}{1}};
                        \node[shape=circle,draw=black,label=above: $31$] (31) at (3,5){};
                        \node[shape=circle,draw=black] (c21) at (4,5){};

                        \draw[line width=.5mm,black,-] (19) edge[-] (27);
                        \draw[line width=.5mm,black,-] (27) edge[-] (31);
                        \draw[line width=.5mm,black,-] (31) edge[-] (c21);

                        \node[shape=circle,draw=black,label=above left: $20$] (20) at (1,4){};
                        \node[shape=circle,draw=black,label=above left: $21$] (21) at (2,4){};
                        \node[shape=circle,draw=black,label=above left: $28$] (28) at (3,4){\textcolor{blue}{1}};
                        \node[shape=circle,draw=black] (c31) at (4,4){};
                        \node[shape=circle,draw=black] (c32) at (5,4){\textcolor{blue}{2}};
                        \node[shape=circle,draw=black] (c33) at (6,4){};
                        \node[shape=circle,draw=black] (c34) at (7,4){};

                        \draw[line width=.5mm,black,-] (20) edge[-] (21);
                        \draw[line width=.5mm,black,-] (21) edge[-] (28);
                        \draw[line width=.5mm,black,-] (28) edge[-] (c31);
                        \draw[line width=.5mm,black,-] (c31) edge[-] (c32);
                        \draw[line width=.5mm,black,-] (c32) edge[-] (c33);
                        \draw[line width=.5mm,black,-] (c33) edge[-] (c34);

                        \node[shape=circle,draw=black,label=above left: $22$] (22) at (1,3){};
                        \node[shape=circle,draw=black,label=above left: $23$] (23) at (2,3){};
                        \node[shape=circle,draw=black,label=above left: $24$] (24) at (3,3){};
                        \node[shape=circle,draw=black,label=above: $25$] (25) at (4,3){};
                        \node[shape=circle,draw=black] (c41) at (5,3){\textcolor{blue}{1}};
                        \node[shape=circle,draw=black] (c42) at (6,3){};

                        \draw[line width=.5mm,black,-] (22) edge[-] (23);
                        \draw[line width=.5mm,black,-] (23) edge[-] (24);
                        \draw[line width=.5mm,black,-] (24) edge[-] (25);
                        \draw[line width=.5mm,black,-] (25) edge[-] (c41);
                        \draw[line width=.5mm,black,-] (c41) edge[-] (c42);

                        \node[shape=circle,draw=black,label=above left: $29$] (29) at (3,2){\textcolor{blue}{1}};
                        \node[shape=circle,draw=black] (c61) at (4,2){};
                        \node[shape=circle,draw=black] (c62) at (5,2){\textcolor{blue}{2}};
                        \node[shape=circle,draw=black] (c63) at (6,2){};
                        \node[shape=circle,draw=black] (c64) at (7,2){};

                        \draw[line width=.5mm,black,-] (29) edge[-] (c61);
                        \draw[line width=.5mm,black,-] (c61) edge[-] (c62);
                        \draw[line width=.5mm,black,-] (c62) edge[-] (c63);
                        \draw[line width=.5mm,black,-] (c63) edge[-] (c64);

                        \node[shape=circle,draw=black] (c71) at (5,1){\textcolor{blue}{1}};
                        \node[shape=circle,draw=black] (c72) at (6,1){};
                        \node[shape=circle,draw=black] (c73) at (7,1){};

                        \draw[line width=.5mm,black,-] (c71) edge[-] (c72);
                        \draw[line width=.5mm,black,-] (c72) edge[-] (c73);

                    \draw[line width=.5mm,black,-] (26) edge[-] (27);
                    \draw[line width=.5mm,black,-] (28) edge[bend left,-] (29);
                    \draw[line width=.5mm,black,-] (c32) edge[-] (c41);
                    \draw[line width=.5mm,black,-] (c62) edge[-] (c71);

                \draw [rounded corners,line width=.4mm,color=red](.4,2.5) --
                      (.4,2.5)  [rounded corners] -- (.4,5.59)
                      [rounded corners] -- (1.5,5.59) -- (1.5,6.75)
                      [rounded corners] -- (4.5,6.75) -- (4.5,4.5)
                      [rounded corners] -- (3.5,4.5) -- (3.5,3.70)
                      [rounded corners] -- (4.5,3.70) -- (4.5,2.5)
                      [rounded corners] -- (3.5,2.5) -- (3.5,1.5)
                      [rounded corners] -- (2.4,1.5) -- (2.4,2.5) [rounded corners] -- cycle;

                    \draw (1.8,5.59) edge[line width=.5mm,red,-] (2.2,5.59);
                    \draw (3.1,3.5) edge[line width=.5mm,red,-] (3.45,3.5);
                    \draw (4.8,3.5) edge[line width=.5mm,red,-] (5.2,3.5);
                    \draw (4.8,1.5) edge[line width=.5mm,red,-] (5.2,1.5);

        \end{scope}

            \node[] at (9,6){$C_1$};
            \node[] at (9,5){$C_2$};
            \node[] at (9,4){$C_3$};
            \node[] at (9,3){$C_4$};
            \node[] at (9,2){$C_6$};
            \node[] at (9,1){$C_7$};

			\node[] at (9,7){$\emph{j=2}$};

    \end{tikzpicture}

\medskip
\par\noindent\rule{\textwidth}{0.4pt}
\medskip

     \begin{tikzpicture}
        \begin{scope}[auto, every node/.style={draw,circle,minimum size=1em,inner sep=1},node distance=1cm]
                    \color{black}
                    \tikzset{dot/.style={fill=black,circle}}
                    
            \node[shape=circle,draw=black,label=above: $32$] (32) at (1,4){};
            \node[shape=circle,draw=black,label=above: $34$] (34) at (2,4){\textcolor{blue}{1}};
            \node[shape=circle,draw=black,label=above: $38$] (38) at (3,4){};
            \node[shape=circle,draw=black] (c31) at (4,4){};
            
            \draw[line width=.5mm,black,-] (32) edge[-] (34);
            \draw[line width=.5mm,black,-] (34) edge[-] (38);
            \draw[line width=.5mm,black,-] (38) edge[-] (c31);
            
            \node[shape=circle,draw=black,label=above left: $35$] (35) at (2,3){\textcolor{blue}{1}};
            \node[shape=circle,draw=black] (c41) at (3,3){};
            
            \draw[line width=.5mm,black,-] (35) edge[-] (c41);
            
            \node[shape=circle,draw=black,label=above: $33$] (33) at (1,2){};
            \node[shape=circle,draw=black,label=above: $36$] (36) at (2,2){\textcolor{blue}{1}};
            \node[shape=circle,draw=black,label=above: $39$] (39) at (3,2){};
            \node[shape=circle,draw=black] (c61) at (4,2){};
            
            \draw[line width=.5mm,black,-] (33) edge[-] (36);
            \draw[line width=.5mm,black,-] (36) edge[-] (39);
            \draw[line width=.5mm,black,-] (39) edge[-] (c61);
            
            \node[shape=circle,draw=black,label=above left: $37$] (37) at (2,1){\textcolor{blue}{1}};
            \node[shape=circle,draw=black,label=above: $40$] (40) at (3,1){};
            \node[shape=circle,draw=black] (c71) at (4,1){};
            
            \draw[line width=.5mm,black,-] (37) edge[-] (40);
            \draw[line width=.5mm,black,-] (40) edge[-] (c71);
            
            \draw[line width=.5mm,black,-] (34) edge[-] (35);
            \draw[line width=.5mm,black,-] (36) edge[-] (37);     
        
            \draw (1.8,3.59) edge[line width=.5mm,red,-] (2.2,3.59);
            \draw (1.8,1.59) edge[line width=.5mm,red,-] (2.2,1.59);
         
         \end{scope}
         
            \node[] at (6,4){$C_3$};
            \node[] at (6,3){$C_4$};
            \node[] at (6,2){$C_6$};
            \node[] at (6,1){$C_7$};

			\node[] at (6,5){$\emph{j=3}$};         

\node [below=2cm, align=flush center,text width=8cm] at (5,2)
       {
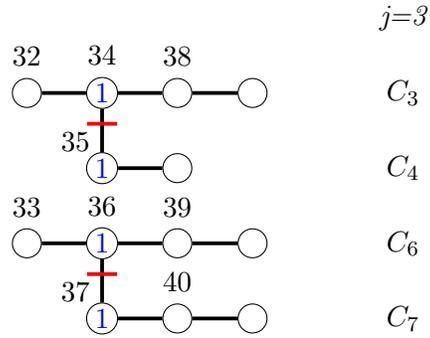
           Forest Labeling Algorithm Example 1
       };

     \end{tikzpicture}

\par\noindent\rule{\textwidth}{0.4pt}
\par\noindent\rule{\textwidth}{0.4pt}

\medskip


     \begin{tikzpicture}
        \begin{scope}[auto, every node/.style={draw,circle,minimum size=1em,inner sep=1},node distance=1cm]
                    \color{black}
                    \tikzset{dot/.style={fill=black,circle}}
                    
            \node[shape=circle,draw=black,label=above: $1$] (1) at (1,5){};
            \node[shape=circle,draw=black,label=above: $3$] (3) at (2,5){};
            \node[shape=circle,draw=black,label=above: $10$] (10) at (3,5){};
			 \node[shape=circle,draw=black,label=above: $11$] (11) at (4,5){};
			 \node[shape=circle,draw=black,label=above: $2$] (2) at (1,4){};
			 \node[shape=circle,draw=black,label=left: $8$] (8) at (3,4){\textcolor{blue}{1}};

            \node[shape=circle,draw=black] (c11) at (5,5){};
            
            \draw[line width=.5mm,black,-] (1) edge[-] (3);
			 \draw[line width=.5mm,black,-](3) edge[-] (10);
 			 \draw[line width=.5mm,black,-] (10) edge[-] (11);
			 \draw[line width=.5mm,black,-] (11) edge[-] (c11);
 			 \draw[line width=.5mm,black,-] (3) edge[-] (2);
 			 \draw[line width=.5mm,black,-] (10) edge[-] (8);
                      
			 \node[shape=circle,draw=black,label=above: $4$] (4) at (1,3){};
			 \node[shape=circle,draw=black,label=above: $6$] (6) at (2,3){};
            \node[shape=circle,draw=black,label=above left: $9$] (9) at (3,3){\textcolor{blue}{1}};
			 \node[shape=circle,draw=black,label=above: $5$] (5) at (1,2){};

            \node[shape=circle,draw=black] (c12) at (4,3){};
            \node[shape=circle,draw=black] (c13) at (5,3){};
            \node[shape=circle,draw=black] (c14) at (6,3){};
            \node[shape=circle,draw=black] (c15) at (3,2){\textcolor{blue}{2}};
            \node[shape=circle,draw=black] (c16) at (4,2){};
            
            \draw[line width=.5mm,black,-] (4) edge[-] (6);
            \draw[line width=.5mm,black,-] (6) edge[-] (9);
            \draw[line width=.5mm,black,-] (9) edge[-] (c12);
            \draw[line width=.5mm,black,-] (c12) edge[-] (c13);
            \draw[line width=.5mm,black,-] (c13) edge[-] (c14);
            \draw[line width=.5mm,black,-] (6) edge[-] (5);
            \draw[line width=.5mm,black,-] (c12) edge[-] (c15);
            \draw[line width=.5mm,black,-] (c13) edge[-] (c16);
            
            \node[shape=circle,draw=black,label=above: $7$] (7) at (2,1){};

            \node[shape=circle,draw=black] (c17) at (3,1){\textcolor{blue}{1}};
            \node[shape=circle,draw=black] (c18) at (4,1){};
            
            \draw[line width=.5mm,black,-] (7) edge[-] (c17);
            \draw[line width=.5mm,black,-] (c17) edge[-] (c18);

            \draw[line width=.5mm,black,-] (9) edge[-] (8);
            \draw[line width=.5mm,black,-] (c15) edge[-] (c17);

            \draw (2.8,3.59) edge[line width=.5mm,red,-] (3.2,3.59);
            \draw (2.8,1.59) edge[line width=.5mm,red,-] (3.2,1.59);
         
         \end{scope}
         
            \node[] at (7,5){$C_1$};
            \node[] at (7,3){$C_2$};
            \node[] at (7,1){$C_3$};

			\node[] at (7,6){$\emph{j=1}$};         

                \draw [rounded corners,line width=.4mm,color=red](.4,.5) --  (.4,5.75)
                      [rounded corners] -- (5.5,5.75) -- (5.5,4.5)
                      [rounded corners] -- (3.5,4.5) -- (3.5,2.5)
                      [rounded corners] -- (2.5,2.5) -- (2.5,.5)
                      [rounded corners] -- (.4,.5) [rounded corners] -- cycle;

     \end{tikzpicture}

\medskip
\par\noindent\rule{\textwidth}{0.4pt}
\medskip

     \begin{tikzpicture}
        \begin{scope}[auto, every node/.style={draw,circle,minimum size=1em,inner sep=1},node distance=1cm]
                    \color{black}
                    \tikzset{dot/.style={fill=black,circle}}
                    
            \node[shape=circle,draw=black,label=above: $14$] (14) at (2,3){};
            \node[shape=circle,draw=black,label=above: $16$] (16) at (3,3){};
            \node[shape=circle,draw=black,label=above: $12$] (12) at (1,2){\textcolor{blue}{1}};
            \node[shape=circle,draw=black,label=above: $15$] (15) at (2,2){};

            \node[shape=circle,draw=black] (c21) at (4,3){};

            \draw[line width=.5mm,black,-] (14) edge[-] (16);
            \draw[line width=.5mm,black,-] (16) edge[-] (c21);
            \draw[line width=.5mm,black,-] (14) edge[-] (12);
            \draw[line width=.5mm,black,-] (15) edge[-] (16);

            \node[shape=circle,draw=black,label=above left: $13$] (13) at (1,1){\textcolor{blue}{1}};

            \node[shape=circle,draw=black] (c22) at (2,1){};

            \draw[line width=.5mm,black,-] (13) edge[-] (c22);

            \draw[line width=.5mm,black,-] (12) edge[-] (13);

            \draw (.8,1.59) edge[line width=.5mm,red,-] (1.2,1.59);

         \end{scope}

            \node[] at (5,3){$C_2$};
            \node[] at (5,1){$C_3$};

			\node[] at (5,4){$\emph{j=2}$};         

\node [below=2cm, align=flush center,text width=8cm] at (5,2)
       {
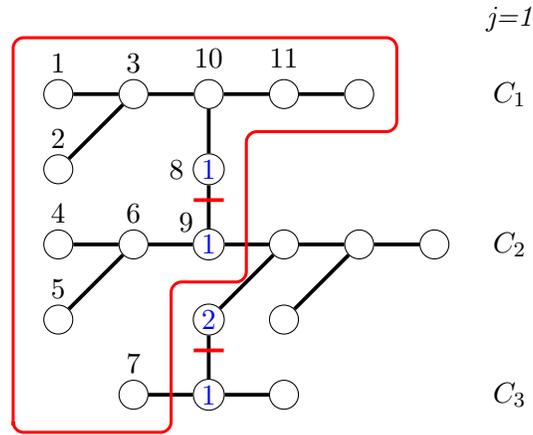
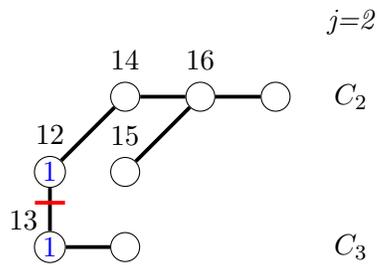
           Forest Labeling Algorithm Example 2
       };

     \end{tikzpicture}

\par\noindent\rule{\textwidth}{0.4pt}

\end{center}

\section{Grundy domination of strong products of graphs}

We first note that to solve Conjecture \ref{conj}, we may assume that $G$ and $H$ are connected graphs, since strong products of graphs have components corresponding to the components of the factor graphs.

To prove our main result, we need to introduce some notation and an essential lemma.

For any graphs $G$ and $H$, choose a maximum legal sequence $D$ of $G\boxtimes H$. For some $v\in V(G)$ let $D_v=\widehat{D}\cap (\{v\}\times H)$ and $F_v$ be the set of vertices of $\widehat{D}$ which have a footprint in $H^{v}$. More specifically, for any $u\in N[v]$, let $D_u(F_v)$ be the set of vertices in $D_u$ which have a footprint in $H^v$.

\begin{lem}\label{footint}
For any graphs $G$ and $H$, and any $v\in V(G)$, $|F_v|\le \gamma_{gr}(H)$.
\end{lem}
\begin{proof}
Let $D$ be a maximum legal sequence of $G\boxtimes H$. Say $N_G(v)=\{u_1,\dots, u_k\}$. Notice that $|F_v|=|D_{u_1}(F_v) \cup \dots \cup D_{u_k}(F_v) \cup D_v(F_v)|$. Choose any vertex $x\in F_v$ and without loss of generality, suppose $x=(u_1,h)$ for $h\in V(H)$. Notice that $x$ has the same footprint in $H^v$ as $(v,h)$. This means that we may project $D_{u_1}(F_v)\cup \dots \cup D_{u_k}(F_v) \cup D_v(F_v)$ onto $H$ and produce a legal set of $H$. Such a legal set has size at most $\gamma_{gr}(H)$ and the lemma is proven. 
\end{proof}

\begin{thm}\label{main}
For any tree $G$ and graph $H$,
\[\gamma_{gr}(G \boxtimes H) = \gamma_{gr}(G)\gamma_{gr}(H).\]
\end{thm}

\begin{proof}
By inequality \eqref{ineq1}, it suffices to show that \[\gamma_{gr}(G \boxtimes H) \le \gamma_{gr}(G)\gamma_{gr}(H).\]

We induct on the order $n$ of $G$. The statement is trivial for $n=1$. Since every tree contains at least two leaf vertices, and leaves are simplicial, we may choose a leaf $v$ and consider the tree $T=G-\{v\}$. Notice that if $\cP$ is a minimum caterpillar partition for $G$ and $\cP'$ is a minimum caterpillar partition for $T$, then $|\cP|-1\le |\cP'|\le |\cP|$. If $|\cP'|= |\cP|$, we may apply Proposition \ref{simplicial}, the inductive hypothesis, and Theorem \ref{forest} to show

\begin{align*}
\gamma_{gr}(G \boxtimes H) & \le \gamma_{gr}(H) + \gamma_{gr}\Big{(}(G-\{v\})\boxtimes H\Big{)}\\
& \le \gamma_{gr}(H) + \gamma_{gr}(G-\{v\})\gamma_{gr}(H)\\
& \le \gamma_{gr}(H) + \Big{(}(n-1) - |\cP'|\Big{)}\gamma_{gr}(H)\\
& \le (n - |\cP'|) \gamma_{gr}(H)= (n - |\cP|) \gamma_{gr}(H)\\
& = \gamma_{gr}(G)\gamma_{gr}(H).
\end{align*}

We may now suppose that for any leaf $v$, $|\cP'|=|\cP|-1$, which means that $G$ is caterpillar-critical. By Lemma \ref{cc}, we may have chosen $\cP$ as a minimum caterpillar partition of $G$ so that $P_5=v_1-v_2-v_3-v_4-v_5$ is a leaf caterpillar in $\cP$ or so that $P_2=u_1-u_2$ is a leaf caterpillar. First suppose we chose such a minimum partition $\cP$ with leaf caterpillar $C=P_5$. Let $v_6$ be the neighbor of $v_3$ which is not on $C$.

Let $D'$ be a maximum legal sequence of $G \boxtimes H$ with the maximum number of entries from $\{v_1\}\times H$. For all such sequences $D'$, let $D$ be one with the maximum number of entries from $\{v_5\}\times H$.

\begin{claim}\label{footC}
Every vertex $v\in D\cap(\{v_2\}\times H)$ must only footprint vertices of $\{v_3\}\times H$ and every vertex $v'\in D\cap(\{v_4\}\times H)$ must only footprint vertices of $\{v_3\}\times H$.
\end{claim}
\begin{proof}
If $v$ footprints $x\in \{v_1\}\times H$, then say $v=(v_2,h)$, notice that $(v_1,h)$ footprints $x$ and may replace $v$, and we still retain a maximum legal sequence of $G\boxtimes H$. However, now the resulting sequence contains more vertices in $\{v_1\}\times H$ than $D$ which contradicts the choice of $D$. 

Next, suppose $(v_2,h)$ footprints $(v_2,h')$ for some $h$ and $h'$ in $V(H)$. Since $\widehat{D}$ is a legal set and dominates $\{v_1\}\times H$, some vertex $y$ of $\widehat{D}$ must dominate $(v_1,h')$, and thus, $y$ dominates $(v_2,h')$. This means that either $y$ or $(v_2,h)$ footprint both $(v_2,h')$ and $(v_1,h')$. 

The symmetric argument applies to $v'$.
\end{proof}

Using our notation from this section, we note that by Lemma \ref{footint}, 
\begin{align}\label{lasteq1}
|D_{v_2}(F_{v_3})|+|D_{v_3}(F_{v_3})|+|D_{v_4}(F_{v_3})|+|D_{v_6}(F_{v_3})|=|F_{v_3}|\le \gamma_{gr}(H).
\end{align}

By Claim \ref{footC},
\begin{align}\label{lasteq2}
D_{v_1}=F_{v_1}, \, D_{v_5}=F_{v_5}, \, D_{v_2}=D_{v_2}(F_{v_3}),\, D_{v_4}=D_{v_4}(F_{v_3}).
\end{align}

We apply inequality \eqref{lasteq1} and identities \eqref{lasteq2} to show

\begin{align*}
|\widehat{D}|&= |D_{v_1}\cup D_{v_2} \cup D_{v_3} \cup D_{v_4} \cup D_{v_5}| + |D\cap ((G-C)\boxtimes H)|\\
&\le |D_{v_1}|+|D_{v_5}|+|D_{v_2}(F_{v_3})|+|D_{v_4}(F_{v_3})|+|D_{v_3}|+|D_{v_6}(F_{v_3})|\\
&+|D\cap ((G-C)\boxtimes H)|\\
&= |D_{v_1}|+|D_{v_5}|+\Big{[}|D_{v_2}(F_{v_3})|+|D_{v_4}(F_{v_3})|+|D_{v_6}(F_{v_3})|\\
&+|D_{v_3}(F_{v_3})|\Big{]}+|D_{v_3}-D_{v_3}(F_{v_3})|+|D\cap ((G-C)\boxtimes H)|\\
&\le 4\gamma_{gr}(H)+\gamma_{gr}((G-C)\boxtimes H)\\
&\mbox{(which we bound by applying the inductive hypothesis to $G-C$)}\\
&\le 4\gamma_{gr}(H)+\gamma_{gr}((G-C))\gamma_{gr}(H)\\
&\mbox{(and applying Theorem \ref{forest})}\\
&=4\gamma_{gr}(H)+(|V(G)|-5-|\cP|+1)\gamma_{gr}(H)\\
&=(|V(G)|-|\cP|)\gamma_{gr}(H) = \gamma_{gr}(G)\gamma_{gr}(H)
\end{align*}

We must now consider the case where no $P_5$ leaf caterpillar exists and $\cP$ is chosen as a minimum caterpillar partition of $G$ so that $C=P_2=u_1-u_2$ is a leaf caterpillar. By Lemma \ref{ccleaf}, we may assume that $C_1=P_5=v_1-v_2-v_3-v_4-v_5$ is a neighbor of $C$ with $u_2$ and $v_3$ as branch vertices. Let $v_6$ be the neighbor of $v_3$ which is not on $C$ or $C_1$. 

Let $D''$ be a maximum legal sequence of $G\boxtimes H$ with the maximum number of entries from $\{v_1\}\times H$. For all such sequences $D''$, let $D'$ be one with the maximum number of entries in $\{v_5\}\times H$, and for all such sequences $D'$, let $D$ be the one with the maximum number of entries in $\{u_1\}\times H$.

\begin{claim}\label{footCC}
Every vertex $v\in D\cap(\{v_2\}\times H)$ must only footprint vertices of $\{v_3\}\times H$, $v'\in D\cap(\{v_4\}\times H)$ must only footprint vertices of $\{v_3\}\times H$, and $v''\in D\cap(\{u_2\}\times H)$ must only footprint vertices of $\{v_3\}\times H$.
\end{claim}

The proof of Claim \ref{footCC} is identical to that of Claim \ref{footC} with the argument for $v$ repeated for $u_2$.

\medskip

By Lemma \ref{footint},
\begin{align}\label{reallylasteq1}
&|D_{v_2}(F_{v_3})|+|D_{v_3}(F_{v_3})|+|D_{v_4}(F_{v_3})|+|D_{u_2}(F_{v_3}|+|D_{v_6}(F_{v_3})|\\
&=|F_{v_3}|\le \gamma_{gr}(H).\nonumber
\end{align}

By Claim \ref{footCC}
\begin{align}\label{reallylasteq2}
D_{v_1}=F_{v_1},\, D_{v_5}=F_{v_5},\, D_{v_2}=D_{v_2}(F_{v_3}),\, D_{v_4}=D_{v_4}(F_{v_3}),\, D_{u_2}=D_{u_2}(F_{v_3}).
\end{align}

We apply inequality \eqref{reallylasteq1} and identities \eqref{reallylasteq2} to show

\begin{align*}
|\widehat{D}|&= |D_{u_1}\cup D_{u_2}\cup D_{v_1}\cup D_{v_2} \cup D_{v_3} \cup D_{v_4} \cup D_{v_5}| + |D\cap ((G-C-C_1)\boxtimes H)|\\
&\le |D_{u_1}|+|D_{v_1}|+|D_{v_5}|+|D_{u_2}(F_{v_3})|+|D_{v_2}(F_{v_3})|+|D_{v_4}(F_{v_3})|+|D_{v_3}|\\
&+|D_{v_6}(F_{v_3})|+|D\cap ((G-C-C_1)\boxtimes H)|\\
&= |D_{u_1}|+|D_{v_1}|+|D_{v_5}|+\Big{[}|D_{u_2}(F_{v_3})|+|D_{v_2}(F_{v_3})|+|D_{v_4}(F_{v_3})|+|D_{v_6}(F_{v_3})|\\
&+|D_{v_3}(F_{v_3})|\Big{]}+|D_{v_3}-D_{v_3}(F_{v_3})|+|D\cap ((G-C-C_1)\boxtimes H)|\\
&\le 5\gamma_{gr}(H)+\gamma_{gr}((G-C-C_1)\boxtimes H)\\
&\mbox{(which we bound by applying the inductive hypothesis to $G-C-C_1$)}\\
&\le 5\gamma_{gr}(H)+\gamma_{gr}((G-C-C_1))\gamma_{gr}(H)\\
&\mbox{(and applying Theorem \ref{forest})}\\
&=5\gamma_{gr}(H)+(|V(G)|-7-|\cP|+2)\gamma_{gr}(H)\\
&=(|V(G)|-|\cP|)\gamma_{gr}(H) = \gamma_{gr}(G)\gamma_{gr}(H)
\end{align*}

\end{proof}

\begin{cor}
For any forest $G$ and graph $H$,
\[\gamma_{gr}(G \boxtimes H) = \gamma_{gr}(G)\gamma_{gr}(H).\]
\end{cor}

\section{Some results about Grundy dominating sets}

\begin{thm}\label{spanningtree}
For any connected graph $G$, there exists a spanning tree $T$ of $G$ so that $\gamma_{gr}(G)\le \gamma_{gr}(T)$.
\end{thm}
\begin{proof}
We induct on the number of cycles in $G$. Notice that the theorem is trivially true if $G$ is a tree. Assume that $G$ contains $m$ cycles and that the theorem holds for every graph with fewer cycles. Let $C$ be a cycle on vertices $x_1,\dots, x_k$ so that consecutive vertices in the list are adjacent as are $x_1$ and $x_k$. Call the edge of $C$ from $x_i$ to $x_{i+1}$, $e_i$, for $i \in [k-1]$, and the edge from $x_k$ to $x_1$, $e_k$.

We will now show that for some $i\in[k]$, $\gamma_{gr}(G)\le \gamma_{gr}(G-e_i)$, completing the proof.

Let $S$ be a maximal legal set for $G$. Notice that for any edge $e=xy$ of $C$, $\gamma_{gr}(G-e)<\gamma_{gr}(G)$ only if either $x$ footprints only $y$, or $y$ footprints only $x$, in $S$. Indeed, if $x$ does not footprint $y$ and $y$ does not footprint $x$, then $S$ is a legal sequence in $G-e$ which means that $\gamma_{gr}(G)\le \gamma_{gr}(G-e)$. Furthermore, if $x$ footprints $y$ as well as some other vertex $z$, then $y$ cannot footprint $x$ and notice that removing $e$ allows the selection of every vertex of $S$ for a legal sequence $S'$ of $G-e$ to which we may add $y$ if it is not dominated. Again, this means that $\gamma_{gr}(G)\le \gamma_{gr}(G-e)$. The argument is identical if $y$ footprints $x$ and another vertex.

Note further that two vertices may not footprint the same vertex, by definition of a legal sequence. We now consider whether the removal of any $e_i$ results in the statement of the claim. Suppose the claim is false. Without loss of generality, suppose $x_1$ footprints only $x_2$. This means that $x_2$ must footprint only $x_3$, and following this chain of reasoning, $x_i$ must footprint only $x_{i+1}$ for every $i\in [k-1]$. Finally, $x_k$ must footprint only $x_1$. This leads to a contradiction, because some vertex from the cycle must be chosen first from the vertices of $C$ for $S$, and thus cannot be footprinted by any other vertex in the cycle.
\end{proof}

Originally, we hoped that Theorem \ref{spanningtree}, together with Theorem \ref{main}, would allow us to induct on the number of cycles in $G$ to solve (or at least make progress on) the strong product conjecture. However, at this time, we do not know the right incantation to conjure that spell.

\begin{thm}\label{totaldom}
For any connected non-complete graph $G$ with at least one edge, there exists a Grundy dominating set $\widehat{S}$ of $G$, so that the induced subgraph of $\widehat{S}$ in $G$ contains no isolated vertices.
\end{thm}

\begin{proof}
We show that if $v$ is an isolated vertex in the induced subgraph of $\widehat{S}$ in $G$, then we can remove $v$ from $S$ and add another vertex to $S$, possibly at a different point in the sequence, so that the resulting sequence is legal and the new vertex is a neighbor of another vertex of $\widehat{S}$. 

To this end, we choose such a vertex $v$ and note that no vertex of $N(v)$ belongs to $\widehat{S}$. We consider two cases.

\begin{case}\label{case1}
Suppose there exists a vertex $x\in \widehat{S}$ at distance $2$ from $v$. 
\end{case}

Notice that $v$ must footprint itself. Let $u$ be a common neighbor of $v$ and $x$. We form the desired legal sequence $S'$ by removing $v$ from $S$ and adding $u$ as the last vertex in $S'$, so that $u$ footprints $v$.

\begin{case}\label{case2}
Suppose no vertex of $\widehat{S}$ is at distance $2$ from $v$. 
\end{case}

Since $\widehat{S}$ is a dominating set, any vertex $u'$ which is at distance $2$ from $v$, must be a neighbor of at least one vertex of $\widehat{S}$. Let $w$ be the footprinter of $u'$ and let $u$ be a common neighbor of $v$ and $u'$. Notice that $v$ must footprint $u$ since it is the only neighbor of $u$ which is in $\widehat{S}$. Form the legal sequence $S'$ by removing $v$ from $S$ and adding $u'$ as the penultimates vertex and $u$ as the ultimate vertex in $S'$. Notice that $u'$ footprints $u$ and $u$ footprints $v$ in $S'$. However, we now have $|\widehat{S}'|=|\widehat{S}|+1$ which is impossible, so this case cannot occur.

\end{proof}

Theorem \ref{totaldom} implies that every connected non-complete graph has a Grundy dominating set which is a total dominating set.

\section{Acknowledgements}

We would like to express our thanks to the anonymous referees who helped us find and correct the flaws in this paper, and improved its readability.

 \bibliographystyle{plain}

\begin{thebibliography}{MMM}



\bibitem{BBGKKMPTV}B.~Bre\v{s}ar, Cs.~Bujt\'{a}s, T.~Gologranc, S.~Klav\v{z}ar, G.~Ko\v{s}mrlj, T.~Marc, B.~Patk\'{o}s, Zs.~Tuza, M.~Vizer, \emph{On Grundy total domination in product graphs}, Discuss. Math. Graph Theory, to appear, DOI: 10.7151/dmgt.2184.

\bibitem{BBGKKPTV1} B.~Bre\v{s}ar, Cs.~Bujt\'{a}s, T.~Gologranc, S.~Klav\v{z}ar, G.~Ko\v{s}mrlj, B.~Patk\'{o}s, Zs.~Tuza, M.~Vizer, \emph{Dominating sequences in grid-like and toroidal graphs}, Electron. J. Combin. 23(4): P4.34 (2016).

\bibitem{BBGKKPTV2} B.~Bre\v{s}ar, Cs.~Bujt\'{a}s, T.~Gologranc, S.~Klav\v{z}ar, G.~Ko\v{s}mrlj, B.~Patk\'{o}s, Zs.~Tuza, M.~Vizer, \emph{Grundy dominating sequences and zero forcing sets}, Discrete Optim. 26 : 66–77 (2017).

\bibitem{BDGHHKR}B.~Bre\v{s}ar, P.~Dorbec, W.~Goddard, B.~Hartnell, M.~Henning, S.~Klav\v{z}ar, D.~Rall, \emph{Vizing's conjecture: a survey and recent results}, J. Graph Theory, Vol. 69 (1): 46-76 (2012).

\bibitem{BGK}B.~Bre\v{s}ar, T.~Gologranc, and T.~Kos, \emph{Dominating sequences under atomic changes with applications in Sierpi\'{n}ski and interval graphs}, Appl. Anal. Discr. Math., Vol. 10 (2): 518-531 (2016).

\bibitem{BGMRR} B.~Bre\v{s}ar, T. Gologranc, M. Milani\v{c}, D.~F.~Rall, R.~Rizzi, Dominating sequences in graphs, Discrete Math. 336: 22–36 (2014).

\bibitem{BHR} B.~Bre\v{s}ar, M.~A.~Henning, D.~F.~Rall, Total dominating sequences in graphs, Discrete Math. 339: 1665–1676 (2016).

\bibitem{V1} V.~G.~Vizing. \emph{Some unsolved problems in graph theory}, Uspehi Mat. Nauk, 23(6 (144)): 117-134 (1968).



  
 \end{thebibliography}
 
 \end{document}